\documentclass{amsart}
\usepackage{times, stmaryrd, verbatim, enumerate, simplewick}
\usepackage{ amssymb, mathrsfs} 

\RequirePackage{color}[1999/02/16 v1.0i Standard LaTeX Color (DPC)]

\newif\ifcomments

\let\newComments\newKibitzer

\newComments\GM{GM}{red}
\newComments\SB{SB}{blue}

\newtheorem{Thm}{Theorem}
\newtheorem{theorem}{Theorem}[section]

\newtheorem{Lem}[theorem]{Lemma}

\newtheorem{corollary}{Corollary}

\newtheorem{rmk}{Remark}[section]

\numberwithin{equation}{section}

\newcommand{\bb}{\mathbb}

\newcommand{\mr}{\mathrm}
\newcommand{\frk}{\mathfrak}
\newcommand{\tr}{\mr{tr}\,}
\newcommand{\mrd}{\mathrm{d}}

\newcommand*{\suchthat}[1]{\left|\vphantom{#1}\right.}

\usepackage{hyperref}
\begin{document}

\title{The classical dynamic symmetry for the $\mathrm{U}(1)$-Kepler problems}
\author{Sofiane Bouarroudj}

\address{Division of Science and Mathematics, New York University Abu Dhabi, Po Box 129188, Abu Dhabi, United Arab Emirates.}
\email{sofiane.bouarroudj@nyu.edu}

\author{Guowu Meng}
\address{Department of Mathematics, Hong Kong Univ. of Sci. and
Tech., Clear Water Bay, Kowloon, Hong Kong.}

\email{mameng@ust.hk}
\thanks{The authors were supported by the Hong Hong Research Grants Council under RGC Project No. 16304014; SB was also supported by the grant NYUAD-065.}


\date{\today}



\begin{abstract}
For the Jordan algebra of hermitian matrices of order $n\ge 2$, we let $X$ be its submanifold consisting of rank-one semi-positive definite elements. The composition of the cotangent bundle map $\pi_X$: $T^*X\to X$ with the canonical map $X\to \mathbb{C}P^{n-1}$ (i.e., the map that sends a hermitian matrix to its column space), pulls back the K\"{a}hler form of the Fubini-Study metric on $\mathbb{C}P^{n-1}$ to a real closed differential two-form $\omega_K$ on $T^*X$. Let $\omega_X$ be the canonical symplectic form on $T^*X$ and $\mu$ be a real number. A standard fact says that $\omega_\mu:=\omega_X+2\mu\,\omega_K$ turns $T^*X$ into a symplectic manifold, hence a Poisson manifold with Poisson bracket $\{\, ,\,\}_\mu$. 

In this article we exhibit a Poisson realization of the simple real Lie algebra $\mathfrak {su}(n, n)$ on the Poisson manifold $(T^*X, \{\, ,\,\}_\mu)$, i.e., a Lie algebra homomorphism from $\mathfrak {su}(n, n)$ to $\left(C^\infty(T^*X, \mathbb R), \{\, ,\,\}_\mu\right)$. Consequently one obtains the Laplace-Runge-Lenz vector for the classical $\mathrm{U}(1)$-Kepler problem with level $n$ and magnetic charge $\mu$. Since the McIntosh-Cisneros-Zwanziger-Kepler problems (MICZ-Kepler Problems) are the $\mathrm{U}(1)$-Kepler problems with level $2$, the work presented here is a direct generalization of the work by A. Barut and G. Bornzin [ {\em J. Math. Phys.} {\bf 12} (1971), 841-843] on the classical dynamic symmetry for the MICZ-Kepler problems.

\end{abstract}

\maketitle
 
\section {Introduction}
Let $\mathfrak g$ be a real Lie algebra. A {\bf Poisson realization} of $\mathfrak g$ on a Poisson manifold $M$ is a Lie algebra homomorphism from $\mathfrak g$ to $\left(C^\infty(M, \mathbb R), \{\, ,\,\}\right)$. It has been known for more than 40 years that $\mathfrak{so}(2,4)$ has a Poisson realization on $M=T^*\mathbb R^3_*$ from which one can reproduce the Kepler problem --- the mathematical model for the simplest solar system.  Here $\mathbb R^3_*:=\mathbb R^3\setminus\{\mathbf 0\}$ is the configuration space for the Kepler problem. As far as we know, this Poisson realization, more precisely its quantized form, was initially discovered \footnote{ For the prehistory of this important discovery about the Kepler problem, one may consult Footnote 2 in Ref. \cite{Barut67}.}  by A.O. Barut and H. Kleinert \cite{Barut67} in 1967.  

A discovery made by H. McIntosh and A. Cisneros \cite{MC70} and independently by D. Zwanziger \cite{Z68} says that the Kepler problem belongs to a family of dynamic models which share the characteristic feature of the Kepler problem, such as the existence of an analogue of the Lapace-Runge-Lenz vector.  These models, refereed to as the {\bf MICZ-Kepler problems} (or MIC-Kepler problems) in the literature, are indexed by a real parameter $\mu$ (called the magnetic charge) with the Kepler problem corresponding to $\mu =0$.    
 
Soon after the discovery of the MICZ-Kepler problems, it was realized that the aforementioned Poisson realization of $\mathfrak{so}(2,4)$ on the phase space of the Kepler problem has an analogue for each MICZ-Kepler problem. Indeed, an explicit quantized form of these Poisson realizations are given by Eqns (A1) and (A14) in Ref. \cite{Barut71}.  

In the literature each such individual Poisson realization of $\mathfrak{so}(2,4)$ is referred to as the {\bf classical dynamic symmetry} for the corresponding MICZ-Kepler problem. Formally, given a dynamic problem $P$ whose phase space is a Poisson manifold $M$,  a Poisson realization $\mathcal R$ of certain real Lie algebra $\frak g$ on $M$, if it exists, is called the classical dynamic symmetry for $P$ provided that $P$ and its solutions can be completely derived from $\mathcal R$. In this sense, the classical dynamic symmetry for the isotropic oscillator in dimension $n$, with $\mathfrak g=\mathfrak{sp}_{2n}(\mathbb R)$, is  also known in the literature. More recently, the classical dynamic symmetry for the magnetized Kepler problems in odd dimension $n=2k+1$, with $\mathfrak g=\mathfrak{so}(2, 2k+2)$, is explicitly given in Ref. \cite{meng2013b}.

Very recently the classical $\mathrm{U}(1)$-Kepler problems have been introduced by the second author \cite{meng2014}, along with their trajectory analysis via an idea originated from Levi-Civita \cite{Levi-Civita1920}. This family of models is indexed by two parameters: an integral parameter $n \ge 2$ and a real parameter $\mu$, and its subfamily with $n=2$ is precisely the family of MICZ-Kepler problems. So it is natural for us to extend the classical dynamic symmetry analysis from the MICZ-Kepler problems to the $\mathrm{U}(1)$-Kepler problems. 

\vskip 10pt
In Section \ref{Jordan algebra} we shall give a quick review of Euclidean Jordan algebras \cite{PJordan33}. This review is primarily based on the book by J. Faraut and A. Kor\'{a}nyi \cite{FK94}. In Section \ref{dynamic symmetry} the classical dynamic symmetry for the Jordan-algebra-based generalized (unmagnetized) Kepler problems \cite{meng2013}, is given. In principle, this classical dynamic symmetry can be deduced from its quantized version in Ref. \cite{meng2013},  but we shall give it a direct verification.  In section \ref{symmetry} we present a family of Poisson realizations for $\frak{su}(n,n)$, summarized in Theorem \ref{main}.  The proof of this theorem is very length, so Section \ref{S:proof} is devoted to it exclusively. As a consequence of this theorem,  we obtain the Laplace-Runge-Lenz vector for any $\mathrm{U}(1)$-Kepler problem. In the final section we describe and prove some quadratic relations concerning this family of Poisson realizations, summarized in Theorem \ref{QRelations}. As a corollary of this last theorem, for each $\mathrm{U}(1)$-Kepler problem, we derive a formula connecting its Hamiltonian to its angular momentum and its Laplace-Runge-Lenz vector, generalizing the formula given by Eq. (2.8) of Ref. \cite{meng2012}.

 \section{Euclidean Jordan algebras}\label{Jordan algebra}
 Let $V$ be a simple Euclidean Jordan algebra, which means that $V$ is both a simple Jordan algebra and an Euclidean vector space such that the Jordan multiplication by any element $u\in V$, an endomorphism on $V$ which is denoted by $L_u$, is self-adjoint with respect to its inner product $\langle \, | \, \rangle$. In our convention the identity element $e$ of $V$ is assumed to be a unit vector, so
 \[
 \langle u | v \rangle ={1\over \rho}\tr (uv)
 \] where $\rho$ is the rank of $V$, $uv$ is the Jordan multiplication of $u$ with $v$, $\tr$ means the trace. To say $V$ is a Jordan algebra means that the bilinear map $(u, v)\mapsto uv$ is symmetric and satisfies the Jordan identity: $L_u\circ L_{u^2}=L_{u^2}\circ L_u$ for any $u\in V$.

In the following we shall identify $V$ with $V^*$ via this map:
\[
u\in V \mapsto \langle u | \, \rangle: V\to \mathbb R.
\]
 
We shall use $S_{uv}$ to denote the endomorphism $[L_u, L_v]+L_{uv}$ for each $u, v\in V$. Note that $S_{ue}=S_{eu}=L_u$. Let us denote $S_{uv}(w)$ by $\{uvw\}$, then we have
\[
[S_{uv}, S_{zw}]=S_{\{uvz\}w}-S_{z\{vzw\}},
\] so these $S_{uv}$ span a real Lie algebra, referred to as the \emph{structure algebra} for $V$, and denoted by $\mathfrak{str}(V)$. The \emph{conformal algebra} of $V$, denoted by $\mathfrak{co}(V)$, is an extension of the structure algebra $\mathfrak{str}(V)$. As a real vector space, we have
\[
\mathfrak{co}(V)=V\oplus \mathfrak{str}(V)\oplus V^*.
\]
By writing $z\in V$ as $X_z$, $\langle w\mid\; \rangle\in V^*$ as $Y_w$, the commutation relations on $\mathfrak{co}(V)$ can be written as follow:
for $u$, $v$, $z$, $w$ in $V$, 
\begin{eqnarray}
\left\{\begin{matrix}[X_u, X_v] =0, \quad [Y_u, Y_v]=0, \quad [X_u,
Y_v] = -2S_{uv},\cr\\ [S_{uv},
X_z]=X_{\{uvz\}}, \quad [S_{uv}, Y_z]=-Y_{\{vuz\}},\cr\\
[S_{uv}, S_{zw}] = S_{\{uvz\}w}-S_{z\{vuw\}}.
\end{matrix}\right.\nonumber
\end{eqnarray}
Note that, when the Jordan algebra is $\Gamma(3)$: $\mathbb R\oplus \mathbb R^3$ (a linear subspace of the real Clifford algebra $\mathrm{Cl}(\mathbb R^3, \mbox{dot product})$) with the product being the symmetrized Clifford multiplication, we have $\frk{co}(V)=\frk{so}(2,4)$ --- the conformal algebra of the Minkowski space, and $\frk{str}(V)=\frk{so}(1,3)\oplus \bb R$. For the case concerning us, $V=\mathrm{H}_n(\mathbb C)$  and $\frk{co}(V)=\frk{su}(n,n)$. Since $\mathrm{H}_2(\mathbb C)\cong \Gamma(3)$ as Jordan algebra, it is not a surprise that $\frk{su}(2,2)\cong\frk{so}(2,4)$ as Lie algebra.

\section{The classical dynamic symmetry for the generalized (unmagnetized) Kepler problems}\label{dynamic symmetry}
In Ref. \cite{meng2013} the second author introduced the Jordan-algebra-based generalized (unmagnetized) Kepler problems, both the quantum models and the classical models. While a few aspects of the quantum models, such as their bound state problem and the quantum dynamic symmetry, are studied in that reference, not a single aspect of the classical models is studied there. In this section we shall devote our attention to the classical dynamic symmetry. This is an easy aspect for the classical models because it can be deduced quickly from its quantum analogue, but the main point here is to set the stage for the study of the classical dynamic symmetry of $\mr{U}(1)$ Kepler problems as well as the generic (magnetized) Kepler problems in the future. 

\subsection{Kepler cones}
As before $V$ denotes a simple Euclidean Jordan algebra with rank $\rho$. We shall also consider $V$ as an Euclidean space, i.e., a smooth space (i.e. a manifold) with the Riemannian metric
\begin{eqnarray}\label{RiemMetric}
\mathrm{d}s^2 = \langle \mathrm{d}x \mid \mathrm{d}x\rangle. 
\end{eqnarray}Here $x$ is the identity map on $V$, but viewed as a map from the smooth space $V$ to the vector space $V$ so that  $\mathrm{d}x$, being the total differential of this vector-valued smooth function, is a vector-valued differential one-form on the smooth space $V$.

For each positive integer $k$ which is at most $\rho$, we let $\mathcal C_k(V)$ or simply $\mathcal C_k$ be the set of rank $k$ semi-positive elements of $V$. It is a fact \cite{meng2013} that $\mathcal C_k$ is a submanifold of $V$ and the tangent space of $\mathcal C_k$ at a point $x$ is 
\[
\{x\}\times \mathrm{Im} L_x
\] where $\mathrm{Im} L_x$ denotes the image of the linear map $L_x$. Moreover, the structure group of $V$ acts on $\mathcal C_k$ homogeneously, whose cotangent lift is a symplectic action on $T^*\mathcal C_k$. This implies that we have a Poisson realization of the structure algebra $\mathfrak{str}$ on the Poisson manifold $T^*\mathcal C_k$.  The surprise is that, this Poisson realization of $\mathfrak{str}$ can be extended to a Poisson realization of the conformal algebra $\mathfrak{co}$.

Before presenting this Poisson realization of $\mathfrak{co}$ on $T^*\mathcal C_k$, we need to do some preparations. First of all, $T^*\mathcal C_k$ shall be identified with $T\mathcal C_k$ via the Riemannian metric \eqref{RiemMetric}. With this identification understood,  $T\mathcal C_k$ becomes a Poisson manifold. Next, we write the inclusion map 
\[
T\mathcal C_k\hookrightarrow TV=V\times V
\] as $(x, \pi)$, and view both $x$ and $\pi$ as vector-valued smooth functions on $T\mathcal C_k$. Note that, at any point $Q$ of  $T\mathcal C_k$, $x(Q)\in \mathcal C_k$ and
$\pi(Q)\in \mathrm{Im} L_{x(Q)}$. 

We use $q^i$ to denote a system of local coordinates on $\mathcal C_k$, $\partial_{q^i}$ to denote the resulting local tangent frame, and let
\[
g_{ij}:=\langle \partial_{q^i} |\partial_{q^i}\rangle, \quad g:=[g_{ij}], \quad g^{ij}: =(g^{-1})_{ij}, \quad E^i=g^{ij}\partial_{q^i}.
\]
Under the identification of  $T^*\mathcal C_k$ with $T\mathcal C_k$ mentioned early, one can see that the local cotangent frame $(\mathrm{d} q^1,\mathrm{d} q^2, \ldots)$ becomes the local tangent frame $(E^1, E^2, \ldots)$, in terms of which we can write
\[
\pi =p_iE^i.
\] 
Also,  under the natural identification of $T_{x}V$ with $V$, we have
$\partial_{q^i}=\frac{\partial x}{\partial q^i}$. Since $\langle E^j | \partial_{q^i} \rangle =\mathrm{d}q^j(\partial_{q^i} )= \delta^j_i$, 
we know that $ \langle E^i | v \rangle \partial_{q^i} |_{x}$ is the orthogonal projection of $v$ onto 
$\mathrm{Im}L_{x}$. So, if we denote by $\bar v$ the function which maps $x\in\mathcal C_k$ to the orthogonal projection of $v$ onto $\mathrm{Im}L_{x}$, then
\begin{eqnarray}\label{projection}
\langle u | \partial_{q^i}\rangle \langle E^i | v \rangle = \langle u |\bar v\rangle.
\end{eqnarray}
For notational sanity, we use the same notation for both a local function on $\mathcal C_k$ and its pull-back under the tangent bundle map $\tau$: $T\mathcal C_k \to \mathcal C_k$. For example, $q^i$ denotes both a local function on $\mathcal C_k$ and its pullback to  $T\mathcal C_k$. 
\begin{Lem}\label{lem: unmagnetized}
For any vectors $u, v\in V$, viewed as constant vector-valued functions on $T\mathcal C_k$, we have
\begin{eqnarray}\label{basicP}
\left\{
\begin{array}{rcl}
\{\langle u | x\rangle, \langle v | x\rangle\} &= &0,\\
\\
 \{\langle u | x\rangle, \langle v | \pi \rangle\} &= & \langle u | \bar v\rangle, \\
 \\
 \{\langle u | \pi \rangle, \langle v | \pi\rangle\} &= & p_ig^{il} \left\langle  \bar {\bar u}  \suchthat{{\partial x\over \partial q^l}}  {\partial^2 x\over \partial q^j\partial q^l}\right\rangle \langle v |E^j \rangle - \langle u\leftrightarrow v \rangle
\end{array}\right.
\end{eqnarray} where $ \bar {\bar u} =u-\bar u$, and  $\langle u\leftrightarrow v \rangle$ denotes the preceding term with $u$ and $v$ being switched. Consequently, for functions $\tilde u, \tilde v$ on $\mathcal C_k$ whose value at $x\in \mathcal C_1$ is inside $\mbox{Im} L_x$, we have
\begin{eqnarray}\label{SpecialP}
 \left\{\langle \tilde u | \contraction{}{\pi}{\rangle, \langle \tilde v |}{\pi}  \pi \rangle, \langle \tilde v | \pi\rangle\right\} =0.
\end{eqnarray}
Here $\contraction{}{\pi}{\cdots }{\pi} \pi \cdots\pi$ means only the Poisson bracket between the two $\pi$'s is counted. 

\end{Lem}
\begin{proof} 
The verification of the Poisson bracket relations in Eq. \eqref{basicP} is based on the local canonical Poisson bracket relations:
\begin{eqnarray}\label{canonicalP}
\{q^i, q^j\}=0, \quad \{q^i, p_j\}=\delta^i_j, \quad \{p_i, p_j\}=0.
\end{eqnarray} 

Since $x$ depends on $q$ only, we have
\[\{\langle u | x\rangle, \langle v | x\rangle\}=0.\]

Next, we have
\begin{eqnarray}
\{\langle u | x\rangle, \langle v | \pi \rangle\} &= & \{\langle u | x\rangle, p_i \langle v | E^i \rangle\}\cr
&=& \{\langle u | x\rangle, p_i \} \langle v | E^i \rangle \quad\mbox{both $x$ and $E_i$ depend on $q$ only}\cr
&=& \langle u | \partial_{q^i}\rangle \langle E^i | v \rangle\quad \mbox{using Eq. \eqref{canonicalP}}\cr
&=&\langle u |\bar v\rangle \quad\mbox{using Eq. \eqref{projection}.}\nonumber
\end{eqnarray}

Similarly, we have
\begin{eqnarray}
\{\langle u | \pi\rangle, \langle v | \pi \rangle\} &= & \{\langle u | p_iE^i\rangle, \langle v | p_jE^j \rangle\}\cr
&=&p_i\{\langle u | E^i\rangle,  p_j\}\langle v |E^j \rangle + \langle u | E^i\rangle\{p_i, \langle v | E^j \rangle\}p_j\cr
&=& p_i\{\langle u | E^i\rangle,  p_j\}\langle v |E^j \rangle - \langle u\leftrightarrow v \rangle.
\nonumber
\end{eqnarray}  
Since
\begin{eqnarray}
\{\langle u | E^i\rangle,  p_j\}   &= & \{\langle u | g^{il}\partial_{q^l}\rangle,  p_j\}  \cr\cr
&= & \left\{g^{il} \left\langle u \suchthat{{\partial x\over \partial q^l}}  {\partial x\over \partial q^l}\right\rangle,  p_j\right\}  \cr\cr
&= & g^{il} \left\langle u \suchthat{{\partial x\over \partial q^l}}  {\partial^2 x\over \partial q^j\partial q^l}\right\rangle  + {\partial g^{il}\over \partial q^j} \left\langle u \suchthat{{\partial x\over \partial q^l}}  {\partial x\over \partial q^l}\right\rangle \cr\cr
&= & g^{il} \left\langle u \suchthat{{\partial x\over \partial q^l}}  {\partial^2 x\over \partial q^j\partial q^l}\right\rangle  - g^{im} g^{nl} {\partial g_{mn}\over \partial q^j} \left\langle u \suchthat{{\partial x\over \partial q^l}}  {\partial x\over \partial q^l}\right\rangle \cr\cr
&= & g^{il} \left\langle u \suchthat{{\partial x\over \partial q^l}}  {\partial^2 x\over \partial q^j\partial q^l}\right\rangle  - g^{im} {\partial g_{mn}\over \partial q^j} \langle u | E^n\rangle \cr\cr
&= & g^{il} \left\langle u \suchthat{{\partial x\over \partial q^l}}  {\partial^2 x\over \partial q^j\partial q^l}\right\rangle  - g^{im} \partial_{q^j}\left(\left\langle {\partial x\over \partial q^m}\suchthat{{\partial x\over \partial q^m}}{\partial x\over \partial q^n}\right\rangle\right) \langle u | E^n\rangle \cr\cr
&= & g^{il} \left\langle u \suchthat{{\partial x\over \partial q^l}}  {\partial^2 x\over \partial q^j\partial q^l}\right\rangle  - g^{im} \left\langle {\partial^2 x\over \partial q^j\partial q^m}\suchthat{{\partial x\over \partial q^m}}\bar u\right\rangle - \left\langle E^i\suchthat{{\partial x\over \partial q^m}}{\partial^2 x\over \partial q^j\partial q^n}\right\rangle \langle u | E^n\rangle \cr\cr
&=& g^{il} \left\langle  \bar {\bar u} \suchthat{{\partial x\over \partial q^l}}  {\partial^2 x\over \partial q^j\partial q^l}\right\rangle- \left\langle E^i\suchthat{{\partial x\over \partial q^m}}{\partial^2 x\over \partial q^j\partial q^n}\right\rangle \langle u | E^n\rangle, \nonumber
\end{eqnarray}
we have
\begin{eqnarray}
\{\langle u | \pi\rangle, \langle v | \pi \rangle\} &=& p_i\{\langle u | E^i\rangle,  p_j\}\langle v |E^j \rangle - \langle u\leftrightarrow v \rangle\cr
&=& p_i \left(g^{il} \left\langle  \bar {\bar u}  \suchthat{{\partial x\over \partial q^l}}  {\partial^2 x\over \partial q^j\partial q^l}\right\rangle -\left\langle E^i\suchthat{{\partial x\over \partial q^m}}{\partial^2 x\over \partial q^j\partial q^n}\right\rangle \langle u | E^n\rangle\right)\langle v |E^j \rangle - \langle u\leftrightarrow v \rangle\cr
&=& p_ig^{il} \left\langle  \bar {\bar u}  \suchthat{{\partial x\over \partial q^l}}  {\partial^2 x\over \partial q^j\partial q^l}\right\rangle \langle v |E^j \rangle - \langle u\leftrightarrow v \rangle.\nonumber
\end{eqnarray}

\end{proof}

\subsection{The classical dynamic symmetry} We are now ready to state the Poisson realization of the conformal algebra $\mathfrak{co}$ on $T\mathcal C_k$  --- the dynamical symmetry for the generalized Kepler problem with configuration space $\mathcal C_k$.  
\begin{Thm}\label{minor}For any vectors $u$, $v$ in $V$,  define functions 
\begin{eqnarray}
\mathcal X_u:=\langle x| \{\pi u\pi\}\rangle, \quad \mathcal S_{uv}: = \langle S_{uv}(x) | \pi\rangle, \quad \mathcal Y_v:=\langle v | x\rangle 
\end{eqnarray}
on $T\mathcal C_k$. Then,  for any vectors $u$, $v$, $z$, $w$ in $V$,  the following Poisson bracket relations hold:
\begin{eqnarray}
\left\{
\begin{matrix}
\{\mathcal X_u, \mathcal X_v\} =0, \quad \{\mathcal Y_u, \mathcal Y_v\}=0, \quad \{\mathcal X_u,
\mathcal Y_v\} = -2\mathcal S_{uv},\cr\\ 
\{\mathcal S_{uv}, \mathcal X_z\}=\mathcal X_{\{uvz\}}, \quad \{\mathcal S_{uv}, \mathcal Y_z\} = -\mathcal Y_{\{vuz\}},\cr\\
\{\mathcal S_{uv}, \mathcal S_{zw}\} = \mathcal S_{\{uvz\}w}-\mathcal S_{z\{vuw\}}.
\end{matrix}\right.\nonumber
\end{eqnarray}
\end{Thm}
\begin{proof}
The proof is a straightforward computation based on the Poisson bracket relations in Eq. \eqref{basicP}. Due to the Poisson bracket relation in Eq. \eqref{SpecialP}, this proof is really a verbatim copy of the proof for Theorem 3.1 in \cite{meng2013}, so it is skipped here.  
\end{proof}

\begin{rmk} The generalized Kepler problem corresponding to the Poisson realization in Theorem \ref{minor} is the Hamiltonian system with phase space $T\mathcal C_k$, Hamiltonian 
 \[
 H={1\over 2} {\mathcal X_e\over \mathcal Y_e} - {1\over \mathcal Y_e},
 \]
 and Laplace-Runge-Lenz vector
 \[
 \mathcal A_u={1\over 2}\left(\mathcal X_u - \mathcal Y_u {\mathcal X_e\over \mathcal Y_e}\right)+{\mathcal Y_u\over \mathcal Y_e}.
 \]
The interested readers may consult Ref. \cite{meng2014'} for more details on this point.
\end{rmk} 
The following subsection is a detailed demonstration of this remark for the Kepler problem.

\subsection{Example: Kepler problem and future light-cone}
The purpose here is to show explicitly a claim made by the 2nd author in the past:  if $V=\Gamma(3):=\mathbb R\oplus \mathbb R^3$, and $k=1$, the generalized Kepler problem is exactly the Kepler problem. In terms of the standard basis vectors $\vec e_0, \vec e_1, \vec e_2, \vec e_3$, the Jordan multiplication can be determined by the following rules: $\vec e_0$ is the identity element, and
\[
\vec e_i\vec e_j=\delta_{ij}\vec e_0
\] for $i, j >0$. The trace $\tr$: $V\to \mathbb R$ is given by the following rules:
\[
\tr \vec e_0 = 2, \quad \tr \vec e_i =0. 
\] So the inner product on $V$ is the one such that the standard basis is an orthonormal basis.
Since $V$ has rank two, the determinant of $x=x^\mu\vec e_\mu$ is
\[
\det x ={1\over 2} ((\tr x)^2-\tr x^2) = (x^0)^2-(x^1)^2-(x^2)^2-(x^3)^2. 
\] Therefore,
\[
\mathcal C_1=\{x\in V \suchthat{} \det x =0, \tr x > 0\}
\] is precisely the future light-cone in the Minkowski space.  It turns out that $\mathcal C_1$ has a global coordinate $q=(q^1, q^2, q^3)$ with $q^i(x)=x^i$. Since $x(q)=r\vec e_0+\vec r$ where $\vec r = q^i\vec e_i$ and $r$ is the length of $\vec r$, we have
\[
 \partial_{q^i} =\vec e_i +{q^i\over r}\vec e_0, \quad g_{ij} =\delta_{ij}+{q^iq^j\over r^2}, \quad g^{ij} = \delta_{ij}-{q^iq^j\over 2r^2}, \quad E^j = \vec e_j - {q^j\over 2r^2}\vec r+{q^j\over 2r}\vec e_0.
\] Here the first and last identities are understood with the natural identification of $T_x\mathcal C_1$ with $\mathrm{Im}L_x$ in mind. 

Let $\vec p= \sum _ip_i \vec e_i$ and $|\vec p|^2 =\vec p \cdot \vec p$. Since $\pi =p_iE^i= \vec p - \frac{\vec p\cdot \vec r}{2r^2}\vec r +\frac{\vec p\cdot \vec r}{2r}\vec e_0$, then $x\pi =(\vec p\cdot \vec r)\vec e_0 + r\vec p$, therefore
\begin{eqnarray}
\mathcal X_e &= & \langle x | \pi^2\rangle =  \langle x\pi | \pi\rangle\cr
&=&  \left\langle (\vec p\cdot \vec r)\vec e_0 + r\vec p\suchthat { \frac{\vec p\cdot \vec r}{2r}} \vec p - \frac{\vec p\cdot \vec r}{2r^2}\vec r +\frac{\vec p\cdot \vec r}{2r}\vec e_0\right\rangle\cr
&=&r|\vec p|^2.\nonumber
\end{eqnarray} 
Since $\mathcal Y_e=r$, we have the Hamiltonian
\[
H={1\over 2} {\mathcal X_e\over \mathcal Y_e} - {1\over \mathcal Y_e}={1\over 2} |\vec p|^2-{1\over r}.
\]
Similarly, one can compute $\mathcal X_{\vec e_i}=\langle x |\{\pi\vec e_i\pi\}\rangle$ and $\mathcal Y_{\vec e_i}=\langle x| \vec e_i\rangle$ and arrive at
\[
\sum_i\mathcal X_{\vec e_i} \vec e_i= 2(\vec r\cdot \vec p)\vec p - \vec r |\vec p|^2, \quad 
\sum_i  \mathcal Y_{\vec e_i}\vec e_i=\vec r.
\]
Then we arrive at the usual Lapace-Runge-Lenz vector for the Kepler problem:
\begin{eqnarray}
\vec A &:=&\sum_i A_{\vec e_i} \vec e_i = \sum_i \left( {1\over 2}\left(\mathcal X_{\vec e_i} - \mathcal Y_{\vec e_i}{\mathcal X_e\over \mathcal Y_e}\right)+{\mathcal Y_{\vec e_i}\over \mathcal Y_e}\right)\vec e_i\cr
&=& (\vec r\times \vec p)\times \vec p +{\vec r\over r}.\nonumber
\end{eqnarray}
\begin{rmk}
A far as we know, the fact that the Lapace-Runge-Lenz vector owes its existence to the dynamic symmetry was initially pointed out by the second author in Subsection 7.1 of Ref. \cite{meng2011}. 
\end{rmk}
 
\section{The classical dynamical symmetry for the $\mathrm{U}(1)$ Kepler problems}\label{symmetry}
In the remainder of this article  the simple Euclidean Jordan algebra $V$ is assumed to be $\mathrm{H}_n(\mathbb C)$ --- the Jordan algebra of complex hermitian matrices of order $n\ge 2$, and $\mu$ is assumed to be a real number. 

In this case $\mathcal C_1$ is homotopy equivalent to $\mathbb CP^{n-1}$, and a generator of $H^2(\mathcal C_1, \mathbb Z)\cong \mathbb Z$ can be chosen to be the cohomology class of the closed real differential two-form $\omega_K\over 2\pi$ where 
\begin{eqnarray}\label{Kepler form}
\omega_K:=-\mathrm{i}{\tr (x\;\mrd x\wedge \mrd x)\over (\tr x)^3 }
\end{eqnarray} and is called the Kepler form in Ref. \cite{meng2014}. On a topologically trivial coordinate patch, there is a real differential one-form $A=A_i\mrd q^i$ such that 
\[
\omega_K =  \mrd A.
\]

We shall also use $\omega_K$ to denote the pullback of $\omega_K$ under the cotangent bundle projection map $T^*{\mathcal C_1}\to \mathcal C_1$. Let $\omega_{\mathcal C_1}$ be the canonical symplectic form on $T^*\mathcal C_1$ and
\[
\omega_\mu:=\omega_{\mathcal C_1}+2\mu\, \omega_K.
\] On a topologically trivial coordinate patch we have
\[
\omega_\mu =\mrd p_i\wedge \mrd q^i + 2\mu\, \mrd A=\mrd (p_i+2\mu A_i)\wedge\mrd q^i,
\] we conclude that $\omega_\mu$ is a symplectic form on $T^*\mathcal C_1$.  As before we shall identify $T^*\mathcal C_1$ with $T\mathcal C_1$ via the inner product on $V$. In this case elements in $V$ are hermitian matrices, so, for any $u, v\in V$, we have matrix product $u\cdot v$, in terms of which, we have the commutator $[u, v]=u\cdot v-v\cdot u$ and the Jordan product $uv={1\over 2}(u\cdot v+v\cdot u)$.

\begin{Lem}\label{LemmaKey}
Assume that $x\in \mathcal C_1$. Let $u, v\in \mathrm{H}_n(\mathbb C)$ and $L_{u, v}=[L_u, L_v]$. 
\begin{enumerate}[(i)] 
\item We have $ x^2 =\tr x\, x$, consequently
\[
[x, ux]={\tr x\over 2}[x, u], \quad \tr(x[ux, v])={1\over 2}\tr x\; \tr(x[u, v]).
\]

\item The following identities hold: 
\[ \mathrm{i}[u,x]\in \mbox{Im} L_x, \quad  x\cdot u\cdot x =\tr(xu)\;x, \quad L_{u, v}x={1\over 4}[[u,v],x].
\] 
Consequently, we have
\begin{eqnarray}\label{keyID}
{1\over 2}[x, [x, u]]  = (\tr x)xu -\tr(xu)x. 
\end{eqnarray}

\item The following Poisson bracket relations on the symplectic manifold $(T\mathcal C_1, \omega_\mu)$ hold:
\[
\{\langle u | x\rangle, \langle v | x\rangle\}=0, \quad \{\langle u | x\rangle, \langle v | \pi \rangle\}=\langle u | \bar v\rangle, 
\]
and
\[
 \left\{\langle \tilde u | \contraction{}{\pi}{\rangle, \langle \tilde v |}{\pi}  \pi \rangle, \langle \tilde v | \pi\rangle\right\} = -2\mu \mathrm{i}{\tr\left(x[\tilde u,\tilde v)]\right)\over (\tr x)^3}
\] provided that $\tilde u, \tilde v$ are functions on $\mathcal C_1$ whose value at $x\in \mathcal C_1$ is inside $\mbox{Im} L_x$.
\end{enumerate}
\end{Lem}
\begin{proof}
\begin{enumerate}[(i)]
\item Since $x$ has rank $1$ and is diagonalizable, it is clear that $x^2 = \tr x x$. Then
\begin{eqnarray}
[x, ux] &=& \frac{1}{2}(x\cdot (u\cdot x +x\cdot u)-(u\cdot x +x\cdot u)\cdot x )\cr
&=&\frac{1}{2}(x^2\cdot  u-u\cdot x^2 ) =  \frac{\tr x}{2}[x, u] \nonumber
\end{eqnarray}
and
\begin{eqnarray}
 \tr(x[ux, v]) &= &\frac{1}{2}\tr ( x\cdot (u\cdot x\cdot v+x\cdot u\cdot v - v\cdot x\cdot u -v\cdot u\cdot x))\cr
 &=& \frac{1}{2}\tr ( x^2\cdot u\cdot v - x^2\cdot v\cdot u)\quad \mbox{using the fact that $\tr$ is cyclic}\cr
 &=& {1\over 2}\tr x\; \tr(x[u, v]).\nonumber
\end{eqnarray}

\item We may assume that $x$ is the matrix whose $(1,1)$-entry is $1$ and all other entries are zero. Then $\mbox{Im} L_x$ is the set consisting of hermitian matrices whose $(i, j)$-entry is zero if $i, j >1$. It is then clear that the hermitian matrix $\mathrm{i}[u,x]$ is an element of $\mbox{Im} L_x$. 

It is clear that $x\cdot u\cdot x$ is the matrix whose $(1, 1)$-entry is $u_{11}$( the $(1, 1)$-entry of $u$) and all other entries are zero. Since $\tr (xu)=u_{11}$, we have $x\cdot u\cdot x =\tr(xu)\;x$. Consequently
\begin{eqnarray}
\frac{1}{2}[x, [x, u]] &=& \frac{1}{2}\bigl( x\cdot (x\cdot u-u\cdot x)- (x\cdot u-u\cdot x)\cdot x\bigr)\cr
&=& x^2u-x\cdot u\cdot x\cr
&=& (\tr x) xu- \tr(xu)x. \nonumber
\end{eqnarray}

The identity $L_{u, v}x={1\over 4}[[u,v],x]$ actually holds for any $x, u, v$: 
\begin{eqnarray}
L_{u, v}x & = &u(vx) - <u\leftrightarrow v>\cr
&=&  \frac{1}{4}\left( u\cdot (v\cdot x+x\cdot v)+ (v\cdot x+x\cdot v)\cdot u - <u\leftrightarrow v> \right)\cr
&=&  \frac{1}{4}\left( u\cdot v\cdot x+u\cdot x\cdot v+ v\cdot x\cdot u+x\cdot v\cdot u - <u\leftrightarrow v> \right)\cr
&=& \frac{1}{4}\left( [u, v]\cdot x+x\cdot [v, u] \right)\cr
&=& {1\over 4}[[u,v],x].\nonumber
\end{eqnarray}

\item The proof is similar to the proof of Lemma \ref{lem: unmagnetized}. Note that, the local canonical Poission relations \eqref{canonicalP} is now changed to
\[
\{q^i, q^j\}=0, \quad \{q^i, p_j\} = \delta^i_j, \quad \{ p_i, p_j\}= -2\mu\mathrm{i}\frac{\tr \left(x \left[\frac{\partial x}{\partial q^i}, \frac{\partial x}{\partial q^j}\right]\right)}{(\tr x)^3 }.
\] 
So the proof of identities $\{\langle u | x\rangle, \langle v | x\rangle\}=0$ and $\{\langle u | x\rangle, \langle v | \pi \rangle\}=\langle u | \bar v\rangle$ is the same as before.

\underline{Proof of identity} $\left\{\langle \tilde u | \contraction{}{\pi}{\rangle, \langle \tilde v|}{\pi}  \pi \rangle, \langle \tilde v | \pi\rangle\right\} = -2\mu \mathrm{i}{\tr\left(x[\tilde u, \tilde v]\right)\over (\tr x)^3}$. 
\begin{eqnarray}
LHS &=& \langle \tilde u | E^i \rangle \langle \tilde v | E^j\rangle \{p_i, p_j\}+\mbox{other terms}\cr
&=& \langle \tilde u | E^i \rangle \langle \tilde v | E^j\rangle \{p_i, p_j\} + 0\quad \mbox{using Eq. \eqref{SpecialP}}\cr
&=& \langle \tilde u | E^i \rangle \langle \tilde v | E^j\rangle (-2\mu\mathrm{i})\frac{\tr \left(x \left[\frac{\partial x}{\partial q^i}, \frac{\partial x}{\partial q^j}\right]\right)}{(\tr x)^3 }\cr
&=& -2\mu \mathrm{i}{\tr(x[\tilde u, \tilde v])\over (\tr x)^3}\cr
&=& RHS. \nonumber
\end{eqnarray}

\end{enumerate}
\end{proof}

\begin{Thm}\label{main} For any vectors $u$, $v$ in $V:=\mathrm{H}_n(\mathbb C)$,  define functions 
\begin{eqnarray}\label{PoissonRMain}
\left\{
\begin{array}{rcl}
\mathcal X_u &:=&  \langle x|\{\pi u\pi\}\rangle +{n\mu^2\over (\tr x)^2}\tr(xu)-\mu \mathrm{i} {\tr(x[u, \pi ])\over \tr x}\\
\mathcal Y_v &:=& \langle v|x\rangle\\
\mathcal S_{uv} &:=&  \langle S_{uv}(x)| \pi \rangle -\mu \mathrm{i} {\tr(x[u, v])\over 2\, \tr x}
\end{array}\right.
\end{eqnarray}
on $T\mathcal C_1$. Then,  for any vectors $u$, $v$, $z$, $w$ in $V$,  the following Poisson bracket relations hold:
\begin{eqnarray}
\left\{
\begin{matrix}
\{\mathcal X_u, \mathcal X_v\} =0, \quad \{\mathcal Y_u, \mathcal Y_v\}=0, \quad \{\mathcal X_u,
\mathcal Y_v\} = -2\mathcal S_{uv},\cr\\ 
\{\mathcal S_{uv}, \mathcal X_z\}=\mathcal X_{\{uvz\}}, \quad \{\mathcal S_{uv}, \mathcal Y_z\} = -\mathcal Y_{\{vuz\}},\cr\\
\{\mathcal S_{uv}, \mathcal S_{zw}\} = \mathcal S_{\{uvz\}w}-\mathcal S_{z\{vuw\}}.
\end{matrix}\right.\nonumber
\end{eqnarray}
\end{Thm}
The proof of this theorem is a bit complicated, so we leave it to the next section.
\begin{rmk} In view of Ref. \cite{meng2014'}, Theorem \ref{main} implies that the corresponding $\mathrm{U}(1)$ Kepler problem is the Hamiltonian system with phase space $T\mathcal C_1$, Hamiltonian 
 \[
 H={1\over 2} {\mathcal X_e\over \mathcal Y_e} - {1\over \mathcal Y_e} \]
 and Laplace-Runge-Lenz vector
 \[
 \mathcal A_u={1\over 2}\left(\mathcal X_u - \mathcal Y_u {\mathcal X_e\over \mathcal Y_e}\right)+{\mathcal Y_u\over \mathcal Y_e}.
 \]
A simple computation yields $H=\frac{\langle x|\pi^2\rangle}{2r} +{n^2\mu^2\over 2(\tr x)^2} -\frac{n}{\tr x}$, i.e., the Hamiltonian in Definition 1.1 of Ref. \cite{meng2014}.  

\end{rmk}

\section{Proof of Theorem \ref{main}}\label{S:proof}
The proof is heavily dependent on Lemma \ref{LemmaKey}. Theorem \ref{minor} says that these identities hold for the constant terms in $\mu$, so we just need to verify them at higher order terms in $\mu$.
\vskip 5pt
\noindent
\underline{Step zero}: It is clear that $\{\mathcal Y_u, \mathcal Y_v\}=0$.
\vskip 5pt
\noindent
\underline{Step one}: Verify that $\{\mathcal S_{uv}, \mathcal Y_z\} = -\mathcal Y_{\{vuz\}}$. This is easy:
\begin{eqnarray}
\{\mathcal S_{uv}, \mathcal Y_z\}  &=& \{\langle S_{uv}(x)| \pi \rangle -\mu \mathrm{i} {\tr(x[u, v])\over 2\, \tr x}, \langle z|x\rangle\}\cr
&=&  \{\langle S_{uv}(x)| \pi \rangle, \langle z|x\rangle\}\cr
&=& - \mathcal Y_{\{vuz\}}\quad \mbox{no higher order terms in $\mu$ involved here}.\nonumber
\end{eqnarray}
\vskip 5pt
\noindent
\underline{Step two}: Verify that $\{\mathcal X_u,
\mathcal Y_v\} = -2\mathcal S_{uv}$. 

\begin{eqnarray}
\{\mathcal X_u, \mathcal Y_v\}  &=& \{ \langle x|\{\pi u\pi\}\rangle +{n\mu^2\over (\tr x)^2}\tr(xu)-\mu \mathrm{i} {\tr(x[u, \pi ])\over \tr x}, \langle v|x\rangle\}\cr
&=& -2\langle S_{uv}(x)|\pi\rangle - {\mu \mathrm{i}\over \tr x}  \{ \tr(x[u, \pi ]), \langle v|x\rangle\}\cr
&=& -2\langle S_{uv}(x)|\pi\rangle + {\mu \over \tr x}  \{\langle v|x\rangle, \tr(\mathrm{i}[x,u] \pi )\}\cr
&=& -2\langle S_{uv}(x)|\pi\rangle + {\mu \over \tr x} \tr(\mathrm{i}[x,u] v)\quad\mbox{using Lemma \ref{LemmaKey}} \cr
&=& -2\langle S_{uv}(x)| \pi \rangle +\mu \mathrm{i} {\tr(x[u, v])\over \tr x}\cr
&=& -2 \mathcal S_{uv}.\nonumber
\end{eqnarray}
\vskip 5pt
\noindent
\underline{Step three}: Verify that $\{\mathcal S_{uv}, \mathcal S_{zw}\} = \mathcal S_{\{uvz\}w}-\mathcal S_{z\{vuw\}}$. This is a bit involved. 

Let $\mathcal L_u:=\mathcal S_{ue}$ and $\mathcal L_{u, v}: ={1\over 2}(\mathcal S_{uv}- \mathcal S_{vu})$. Then 
\[
\mathcal L_u=\langle ux|\pi\rangle, \quad \mathcal L_{u, v}=\langle L_{u, v}x | \pi\rangle- \mu \mathrm{i} {\tr(x[u, v])\over 2\, \tr x}, \quad 
\mathcal S_{uv}=\mathcal L_{u, v}+\mathcal L_{uv}.
\]
We claim that\begin{eqnarray}
\{\mathcal L_u, \mathcal L_v \} = \mathcal L_{u, v}, \quad \{\mathcal L_{u, v}, \mathcal L_z \} = \mathcal L_{L_{u, v}z}.
\end{eqnarray}
\underline{Proof that $\{\mathcal L_u, \mathcal L_v\} = \mathcal L_{u, v}$}\;:
\begin{eqnarray}
\{\mathcal L_u, \mathcal L_v\} &=&\{\langle ux | \pi\rangle, \langle vx | \pi\rangle  \}\cr
& =& \langle L_{u, v}x | \pi\rangle +\{\langle ux | \contraction{}{\pi}{\rangle,  \langle vx |}{\pi}\pi\rangle,  \langle vx | \pi \rangle  \}\cr
&=&  \langle L_{u, v}x | \pi\rangle  -2\mu \mathrm{i}{\tr(x[ux, vx])\over (\tr x)^3}\cr
&=&   \langle L_{u, v}x | \pi\rangle  - \mu \mathrm{i}{\tr(x[u, v])\over 2\tr x}\cr
&=& \mathcal L_{u, v}.\nonumber
\end{eqnarray}
\underline{Proof that $\{\mathcal L_{u,v}, \mathcal L_z\} = \mathcal L_{L_{u, v}z}$}\; :
\begin{eqnarray}
\{\mathcal L_{u,v}, \mathcal L_z\} &=&\{\langle L_{u, v}x | \pi\rangle  - \mu \mathrm{i}{\tr(x[u, v])\over 2\tr x}, \langle zx | \pi\rangle  \}\cr
&=& \mathcal L_{L_{u, v}z}  + \mu \mathrm{i}\{\langle zx | \pi\rangle, {\tr(x[u, v])\over 2\tr x}\} + \{\langle L_{u, v}x | \contraction{}{\pi}{\rangle,  \langle zx | }{\pi} \pi\rangle,  \langle zx | \pi\rangle\} \cr
&=& \mathcal L_{L_{u, v}z}  + \mu \mathrm{i}\{\langle zx | \pi\rangle, {\tr(x[u, v])\over 2\tr x}\}  -2\mu \mathrm{i}{\tr(x[L_{u, v}x, zx])\over (\tr x)^3} \cr
&=& \mathcal L_{L_{u, v}z}  + \mu \mathrm{i}\left(-{\tr((zx)[u, v])\over 2\tr x} +{\tr(x[u, v])\over 2(\tr x)^2}\tr(zx) \right) \cr
&& -\mu \mathrm{i}{\tr(x[L_{u, v}x, z])\over (\tr x)^2} \cr
&=& \mathcal L_{L_{u, v}z}\nonumber
\end{eqnarray}
provided that
\[
-\tr((zx)[u, v])\tr x + \tr(x[u, v])\tr(zx) 
 -2\tr(x[L_{u, v}x, z])=0,
 \]
 or 
\[
-(\tr x) x[u, v] + \tr(x[u, v])x =2[x, L_{u, v}x],
\]
which is implied by the following identities 
\[ L_{u, v}x={1\over 4}[[u,v],x], \quad {1\over 2}[x, [x, u]]  = (\tr x)xu -\tr(xu)x 
\] in part (ii) of Lemma \ref{LemmaKey}.

We are now ready to prove that \underline{$\{\mathcal S_{uv}, \mathcal S_{zw}\} = \mathcal S_{\{uvz\}w}-\mathcal S_{z\{vuw\}}$}: Since $\mathcal S_{uv}=\mathcal L_{u, v}+\mathcal L_{uv}$ and $
\mathcal S_{zw} = \mathcal L_{z, w}+\mathcal L_{zw}$, we have
\begin{eqnarray}
\{\mathcal S_{uv}, \mathcal S_{zw}\} &=&  \{\mathcal L_{u, v}, \mathcal L_{z, w}\}+ \{\mathcal L_{u, v}, \mathcal L_{zw}\}+ \{\mathcal L_{uv}, \mathcal L_{z, w}\}+ \{\mathcal L_{uv},\mathcal L_{zw}\}\cr\cr
&=&  \{\mathcal L_{u, v}, \{\mathcal L_z, \mathcal L_w\}\}+\mathcal L_{L_{u,v}(zw)}-\mathcal L_{L_{z,w}(uv)}+ \mathcal L_{uv,zw}\cr\cr
&=&  \{\{\mathcal L_{u, v}, \mathcal L_z\}, \mathcal L_w\}+\{\mathcal L_z, \{\mathcal L_{u, v}, \mathcal L_w\}\}+\mathcal L_{L_{u,v}(zw)}-\mathcal L_{L_{z,w}(uv)}+ \mathcal L_{uv,zw}\cr\cr
&=&  \{\mathcal L_{L_{u, v}z}, \mathcal L_w\}+\{\mathcal L_z, \mathcal L_{L_{u, v}w}\}+\mathcal L_{L_{u,v}(zw)}-\mathcal L_{L_{z,w}(uv)}+ \mathcal L_{uv,zw}\cr\cr
&=&  \mathcal L_{L_{u, v}z, w} + \mathcal L_{z, L_{u, v}w}+\mathcal L_{L_{u,v}(zw)}-\mathcal L_{L_{z,w}(uv)}+ \mathcal L_{uv,zw}\cr\cr
&=& \mathcal L_{\{uvz\}, w} -  \mathcal L_{(uv)z, w} - \mathcal L_{z, \{vuw\}}+ \mathcal L_{z, (uv)w}\cr
&& +\mathcal L_{L_{u,v}(zw)}-\mathcal L_{L_{z,w}(uv)}+ \mathcal L_{uv,zw} \cr\cr
&=& \mathcal S_{\{uvz\}w}-\mathcal S_{z \{vuw\}} -  \mathcal L_{(uv)z, w} + \mathcal L_{z, (uv)w}+ \mathcal L_{uv,zw}\cr
&&  -\mathcal L_{\{uvz\}w} + \mathcal L_{z \{vuw\}}+\mathcal L_{L_{u,v}(zw)}-\mathcal L_{L_{z,w}(uv)} \cr\cr
&=& \mathcal S_{\{uvz\}w}-\mathcal S_{z \{vuw\}} \nonumber
\end{eqnarray} 
provided that
\[
 \mu \mathrm{i} {\tr(x[(uv)z, w])\over 2\, \tr x}- \mu \mathrm{i} {\tr(x[z, (uv)w])\over 2\, \tr x}- \mu \mathrm{i} {\tr(x[uv, zw])\over 2\, \tr x} = 0
\]
or
\[
\tr(x[(uv)z, w]) - \tr(x[z, (uv)w]) -  \tr(x[uv, zw]) = 0
\]
or
\[
[x,(uv)z] = [x,z](uv) + [x, uv]z
\] which is clearly true because $[x, ]$ is a derivation.
\vskip 5pt
\noindent
\underline{Step four}: Verify that $\{\mathcal S_{uv}, \mathcal X_z\}=\mathcal X_{\{uvz\}}$. It suffices to verify that $\{\mathcal L_u, \mathcal X_v\}=\mathcal X_{uv}$:
\begin{eqnarray}
\{\mathcal S_{uv}, \mathcal X_z\} &=& \{\mathcal L_{u,v}, \mathcal X_z\}+\{\mathcal L_{uv}, \mathcal X_z\}\cr
&=& \{\{\mathcal L_u,\mathcal L_v\}, \mathcal X_z\}+\{\mathcal L_{uv}, \mathcal X_z\}\cr
&=& \{\{\mathcal L_u,\mathcal X_z\}, \mathcal L_v\} + \{\mathcal L_u,\{\mathcal L_v, \mathcal X_z\}\}+\mathcal X_{(uv)z}\cr
&=&- \mathcal X_{v(uz)} + \mathcal X_{u(vz)} +\mathcal X_{(uv)z}\cr
&=& \mathcal X_{\{uvz\}}.\nonumber
\end{eqnarray}
\underline{Proof that $\{\mathcal L_u, \mathcal X_v\}=\mathcal X_{uv}$}, i.e.,
\[
\left\{\langle ux|\pi\rangle,  \langle x|\{\pi v\pi\}\rangle +{n\mu^2\over (\tr x)^2}\tr(xv)-\mu \mathrm{i} {\tr(x[v, \pi ])\over \tr x}\right\}
\] is equal to 
\[
 \langle x|\{\pi (uv)\pi\}\rangle +{n\mu^2\over (\tr x)^2}\tr(x(uv))-\mu \mathrm{i} {\tr(x[uv, \pi ])\over \tr x}
\] which involve terms up to degree two in $\mu$. Note that there is no need to verify it for terms constant in $\mu$ because of Theorem \ref{minor}. 

\underline{For terms quadratic in $\mu$}, we have to verify that
\[
\left\{\langle ux|\pi\rangle,  {n\mu^2\over (\tr x)^2}\tr(xv)\right\}  -{\mu \mathrm{i} \over \tr x}\left\{\langle ux|\contraction{}{\pi}{\rangle,  \tr(x[v, }{\pi}\pi\rangle,  \tr(x[v, \pi ]\right\} = {n\mu^2\over (\tr x)^2}\tr(x(uv)),
\]
i.e.,
\[
-{n\mu^2\over (\tr x)^2}\tr((ux)v) +2 {n\mu^2\over (\tr x)^3}\tr(xv)\tr(ux)  -{n\mu \mathrm{i} \over \tr x}\left\{\langle ux|\contraction{}{\pi}{\rangle,  \langle[x, v]|}{\pi}\pi\rangle,  \langle[x, v]|\pi\rangle)\right\} = {n\mu^2\over (\tr x)^2}\tr(x(uv))
\]
or
\[
-{n\mu^2\over (\tr x)^2}\tr((ux)v) +2 {n\mu^2\over (\tr x)^3}\tr(xv)\tr(ux)  - {2n\mu^2\over \tr x}{\tr(x[ux, [x, v]])\over (\tr x)^3} = {n\mu^2\over (\tr x)^2}\tr(x(uv))
\]
or
\[
-{n\mu^2\over (\tr x)^2}\tr((ux)v) +2 {n\mu^2\over (\tr x)^3}\tr(xv)\tr(ux)  - {n\mu^2\over (\tr x)^3}\tr(x[u, [x, v]]) = {n\mu^2\over (\tr x)^2}\tr(x(uv))
\]
or
\begin{eqnarray}
-\tr(v(ux))+{2\over \tr x}\tr(vx)\tr(ux)
 -{\tr(x[u, [x, v]])\over \tr x}
=\tr(x(uv))
\end{eqnarray}
or
\begin{eqnarray}
-vx+{2\over \tr x}\tr(vx)x
 -{1\over \tr x}[[x,v],x]
=xv
\end{eqnarray}which is essentially identity \eqref{keyID}. 

\underline{For terms linear in $\mu$}, we have to verify that
\[
2\left\{\langle ux|\contraction{}{\pi}{\rangle,  \langle x|\{}{\pi}\pi\rangle,  \langle x|\{\pi v\pi\}\rangle \right\}  -\mu \mathrm{i}\left\{\langle ux|\pi\rangle, {\tr(x[v, \pi ])\over \tr x}\right\}\suchthat{1\over 2}_{\mbox{no $\pi\pi$ contraction}} =  -\mu \mathrm{i} {\tr(x[uv, \pi ])\over \tr x},
\]
i.e.
\begin{eqnarray}
- \mu \mathrm{i} {\tr(x[(uv), \pi ])\over \tr x} &=&  -4\mu \mathrm{i} {\tr(x[ux, S_{v\pi} (x)])\over (\tr x)^3}-\mu \mathrm{i}{\tr((u\pi)[x, v])\over \tr x}\cr
&&+\mu \mathrm{i} {\tr((ux)[v, \pi ])\over \tr x} - \mu \mathrm{i} {\tr(x[v, \pi ])\over (\tr x)^2}\tr(ux)\nonumber
\end{eqnarray}
or
\begin{eqnarray}
- \mu \mathrm{i} {\tr(x[(uv), \pi ])\over \tr x} &=&  -2\mu \mathrm{i} {\tr(x[u, S_{v\pi} (x)])\over (\tr x)^2}-\mu \mathrm{i}{\tr((u\pi)[x, v])\over \tr x}\cr
&&+\mu \mathrm{i} {\tr((ux)[v, \pi ])\over \tr x} - \mu \mathrm{i} {\tr(x[v, \pi ])\over (\tr x)^2}\tr(ux)\nonumber
\end{eqnarray}
or
\begin{eqnarray}
2\tr(x[u, S_{v\pi} (x)]) &=& - \tr((u\pi)[x, v]) \tr x+\tr((ux)[v, \pi ]) \tr x \cr
&&-  \tr(x[v, \pi ]) \tr(ux) + \tr(x[uv, \pi ]) \tr x\nonumber
\end{eqnarray}
or
\begin{eqnarray}
2[S_{v\pi} (x), x] &=& - \pi[x, v] \tr x+x[v, \pi ] \tr x -  \tr(x[v, \pi ]) x + v[\pi, x ] \tr x.\nonumber
\end{eqnarray}
Expanding the term on the left and combining the 1st and last terms on the right, we arrive at the identity
\begin{eqnarray}
2[v(\pi x), x] +2[(v\pi )x, x]-2[\pi(vx), x] = [\pi v, x]\tr x+x[v, \pi]\tr x-\tr(x[v, \pi])x.\nonumber
\end{eqnarray}
Since $2[(v\pi )x, x]=[v\pi, x]\tr x$, the preceding identity becomes
\begin{eqnarray}
2[L_{v, \pi} x, x]  =x[v, \pi]\tr x-\tr(x[v, \pi])x\nonumber
\end{eqnarray}
or 
\begin{eqnarray}
{1\over 2}[x, [x, [v, \pi]]]  = x[v, \pi]\tr x-\tr(x[v, \pi])x\nonumber
\end{eqnarray} which is implied by the following identities 
\[ L_{u, v}x={1\over 4}[[u,v],x], \quad {1\over 2}[x, [x, u]]  = (\tr x)xu -\tr(xu)x 
\] in part (ii) of Lemma \ref{LemmaKey}.

\vskip 5pt
\noindent
\underline{Step five}: Verify that  $\{\mathcal X_u, \mathcal X_v\} =0$. It suffices to verify that
$\{\mathcal X_u, \mathcal X_e\} =0$:
\begin{eqnarray}
\{\mathcal X_u, \mathcal X_v\} &=&\{\mathcal X_u, \{\mathcal L_v, \mathcal X_e\}\}\cr
&=&\{\{\mathcal X_u,\mathcal L_v\}, \mathcal X_e\} + \{\mathcal L_v, \{\mathcal X_u, \mathcal X_e\}\}\cr
&=& -\{\mathcal X_{uv}, \mathcal X_e\} =0.\nonumber
\end{eqnarray}

\noindent
\underline{Proof that $\{\mathcal X_u, \mathcal X_e\} =0$}, i.e., 
\[
 \left\{ \langle x|\{\pi v\pi\}\rangle +{n\mu^2\over (\tr x)^2}\tr(xv)-\mu \mathrm{i} {\tr(x[v, \pi ])\over \tr x},  \langle x|\pi^2 \rangle +{n\mu^2\over \tr x} \right\} = 0.
\]
This identity has terms up to degree four in $\mu$. Again there is no need to verify the degree zero terms in $\mu$. 

\underline{For terms linear in $\mu$}, we have to verify that 
\begin{eqnarray}
0 &=& 4\{ \langle S_{v\pi}x | \contraction{}{\pi}{\rangle,  \langle \pi x|}{\pi} \pi\rangle,  \langle \pi x|\pi \rangle\} \cr
&&-{\mu \mathrm{i} \over \tr x} \{\tr([x, v] \contraction{}{\pi}{ ),  \langle \pi^2| }{x }\pi ),  \langle \pi^2| x \rangle \} -{2\mu \mathrm{i} \over \tr x}\{\tr([\contraction{}{x}{,v] \pi ),  \langle \pi x|}{\pi}x,v] \pi ),  \langle \pi x|\pi \rangle \}\cr
 && +2\mu \mathrm{i}  {\tr([x,v] \pi )\over (\tr x)^2}\{\tr \contraction{}{x}{,  \langle \pi x|}{\pi }x,  \langle \pi x|\pi \rangle \}\nonumber
\end{eqnarray}
i.e.,
\begin{eqnarray}
0&=& -8\mu \mathrm{i}{\tr(x[S_{v\pi}x, \pi x])\over (\tr x)^3} +{n\mu \mathrm{i} \over \tr x}  \langle \pi^2| [x, v] \rangle\cr
 && -{2\mu \mathrm{i} \over \tr x} \tr([\pi x,v] \pi ]) +2\mu \mathrm{i}  {\tr([x,v] \pi )\over (\tr x)^2} \tr (\pi x)\nonumber
\end{eqnarray}
or
\begin{eqnarray}
0&=& -4\mu \mathrm{i}{\tr(x[S_{v\pi}x, \pi])\over (\tr x)^2} +{n\mu \mathrm{i} \over \tr x}  \langle \pi^2| [x, v] \rangle\cr
 && -{2\mu \mathrm{i} \over \tr x} \tr([\pi x,v] \pi ]) +2\mu \mathrm{i}  {\tr([x,v] \pi )\over (\tr x)^2} \tr (\pi x)\nonumber
\end{eqnarray}
or
\[
 -2\tr(x[S_{v\pi}x, \pi ]) + {1\over 2}\tr x \; \tr( \pi^2 [x, v] )  -  \tr x \tr([\pi x,v] \pi ) 
+ \tr([x,v] \pi ]) \tr (\pi x) =0
\]
or
\[
 -2\tr(x[S_{x\pi}v, \pi ]) + {1\over 2}\tr x \; \tr( [\pi^2, x] v )     -  \tr x \tr([\pi x,v] \pi ) 
+ \tr([x,v] \pi ) \tr (\pi x) =0
\]

or
\[
-2S_{\pi x}([\pi, x]) + {1\over 2}\tr x  [ \pi^2, x]-\tr x [\pi, \pi x]+\tr(\pi x)[\pi, x]=0.
\]
Since $2\pi x =\tr x \pi+\tr \pi x$, we have
\[
-2S_{\pi x}([\pi, x]) + {1\over 2}\tr x  [ \pi^2, x]+{1\over 2}\tr\pi\; \tr x [\pi, x]=0.
\]
So we need to verify that

\[
S_{\pi x}([\pi, x])= {1\over 4}\tr x  [ \pi^2, x]+{1\over 4}\tr\pi\; \tr x [\pi, x]
\]
which can indeed be verified: 
\begin{eqnarray}
S_{\pi x}([\pi, x]) &= & \pi (x[\pi, x])-x(\pi[\pi, x])+(\pi x)[\pi, x]\cr
&=&  {1\over 2}\pi [\pi, x^2]-{1\over 2}x[\pi^2, x]+(\pi x)[\pi, x]\cr
&=&  {1\over 4} [\pi^2, x^2]-{1\over 4}[\pi^2, x^2]+(\pi x)[\pi, x]\cr
&= &(\pi x)[\pi, x]\cr
&=& {1\over 2}(\tr x\; \pi+\tr \pi\; x) [\pi, x]\quad\mbox{$\because \pi$ is a tangent vector}\cr
&=& {1\over 2}\tr x\; \pi [\pi, x]+{1\over 2}\tr \pi\; x [\pi, x]\cr
&=& {1\over 4}\tr x\; [\pi^2, x]+{1\over 4}\tr \pi\; [\pi, x^2]\cr
&=& {1\over 4}\tr x\; [\pi^2, x]+{1\over 4}\tr \pi\;\tr x [\pi, x].\nonumber
\end{eqnarray}

\vskip 5pt
\underline{For terms quadratic in $\mu$}, we have to verify that
\begin{eqnarray}
0 &=&  2\{ \langle S_{v\pi}x|  \contraction{}{\pi}{\rangle, -\tr}{x}\pi\rangle, -\tr x\} { n\mu^2\over (\tr x)^2}\cr
&& +2 \left\{ {n\mu^2\over (\tr x)^2}\tr(\contraction{}{x}{v),  \langle \pi x|}{\pi}xv),  \langle \pi x|\pi \rangle \right\}-4  {n\mu^2\over (\tr x)^3}\tr(xv)\left\{\tr \contraction{}{x}{,  \langle \pi x|}{\pi}x,  \langle \pi x|\pi \rangle \right\}\cr
 && - {2\mu \mathrm{i} \over \tr x}\{\tr([x,v] \contraction{}{\pi}{ ]),  \langle \pi x|}{\pi}\pi ]),  \langle \pi x|\pi \rangle \},\nonumber
\end{eqnarray}
i.e.,
\begin{eqnarray}
0 &=& 2 \tr (S_{v\pi}x){ n\mu^2\over (\tr x)^2}\cr
&& +2 {n\mu^2\over (\tr x)^2}\tr((\pi x)v) -4  {n\mu^2\over (\tr x)^3}\tr(xv) \tr (\pi x)\cr
 &&  - {4n\mu^2 \over \tr x}{\tr(x[[x, v], \pi x])\over (\tr x)^3}\nonumber
\end{eqnarray}
or
\begin{eqnarray}
0 &=& 2 \tr (S_{v\pi}x){ n\mu^2\over (\tr x)^2}\cr
&& +2 {n\mu^2\over (\tr x)^2}\tr((\pi x)v) -4  {n\mu^2\over (\tr x)^3}\tr(xv) \tr (\pi x) \cr
 &&  - {2n\mu^2 \over \tr x}{\tr(x[[x, v], \pi])\over (\tr x)^2} \nonumber
\end{eqnarray}
or

\[
   \tr (S_{v\pi}x)\tr x
 + \tr x \tr((\pi x)v) -2 \tr(xv) \tr (\pi x)
 -\tr(x[[x, v], \pi])=0
 \]
or

\[
   \tr ((v\pi)x)\tr x
 + \tr x \tr((\pi x)v) = 2 \tr(xv) \tr (\pi x) +\tr(x[[x, v], \pi])
 \]
or 
\[
  2 \tr x \tr((\pi x)v) = 2 \tr(xv) \tr (\pi x) +\tr([x, [x, v]] \pi).
 \]
Since ${1\over 2}[x,[x,v]]= \tr x (xv)-\tr(xv)x$, the preceding identity becomes

\[
 2\tr x \tr((\pi x)v) = 2 \tr(xv) \tr (\pi x) +2\tr x \tr ((xv)\pi)-2\tr(xv)\tr (\pi x)
 \] which is trivially true.

\vskip 5pt
\underline{For terms cubic in $\mu$}, we have to verify that
\[
0= \left\{-\mu \mathrm{i} {\tr(x[v, \pi ])\over \tr x},  {n\mu^2\over \tr x} \right\} \quad\mbox{or} \quad
\left\{\tr([x, v] \pi ), \tr x \right\}=0
\]
or \[
0= \tr([x,v])
\] which is trivially true.

There are no terms higher than cubic.

\section{Quadratic Relations}
The main purpose of this section is to show that, in the Poisson realization for the conformal algebra of ${\mathrm H}_n(\mathbb C)$ that we have proved in the preceding section, the generators of the conformal 
algebra 
\[
\mathcal S_{u, v}, \quad \mathcal X_z, \quad \mathcal Y_w
\]
satisfy some quadratic relations. Moreover, these quadratic relation is the consequence of a single one which shall be called the primary quadratic relation. As a consequence, for the corresponding $\mathrm{U}(1)$ Kepler problem, we obtain a formula connecting the Hamiltonian to the angular momentum and the Laplace-Runge-Lenz vector. This formula generalizes the one given by Eq. (2.8) of Ref. \cite{meng2012}.
 
\begin{Thm}\label{QRelations}
Let $e_\alpha$ be an orthonormal basis for ${\mathrm H}_n(\mathbb C)$. In the following we hide the summation sign over $\alpha$ or $\beta$. For the Poisson realization given by Eq. \eqref{PoissonRMain},
we have the following
\begin{enumerate}[(i)]
\item primary quadratic relation
\begin{eqnarray}\label{case1}
{2\over n} \mathcal L_{e_\alpha}^2-\mathcal  L_e^2- \mathcal  X_e  \mathcal Y_e = -\mu^2
\end{eqnarray}
\\
\noindent and secondary quadratic relations
\item $\mathcal  X_{e_\alpha} \mathcal  L_{e_\alpha}=n\mathcal  X_e \mathcal  L_e$, 
$\mathcal  Y_{e_\alpha} \mathcal  L_{e_\alpha} = n\mathcal  Y_e \mathcal  L_e$,
 \\
 \item ${4\over n}\mathcal  L_{e_\alpha, u} \mathcal  L_{e_\alpha}=-\mathcal  X_u\mathcal  Y_e+\mathcal  X_e \mathcal  Y_u$,
\\ 
\item 
$\mathcal  X_{e_\alpha}^2 =n\mathcal  X_e^2$, $\mathcal  Y_{e_\alpha}^2 =n\mathcal  Y_e^2$,
\\
\item ${2\over n} \mathcal  L_{e_\alpha, u} \mathcal  X_{e_\alpha}=-\mathcal  X_u \mathcal  L_e+\mathcal  L_u \mathcal  X_e$, ${2\over n}\mathcal  L_{e_\alpha,u} \mathcal  Y_{e_\alpha}=\mathcal  Y_u \mathcal  L_e-\mathcal  L_u\mathcal  Y_e$,
\\
\item $\mathcal  X_{e_\alpha} \mathcal  Y_{e_\alpha}= n (\mathcal  L_e^2 + \mu^2), 
$
\\
\item $\frac{4}{n^3}\mathcal  L_{e_\alpha, e_\beta}^2=\mathcal  X_e \mathcal  Y_e-\mathcal  L_e^2+\frac{n-2}{n}\mu^2$.
\end{enumerate}

\end{Thm}
\begin{proof}
\begin{enumerate}[(i)]
\item
Since $\mathcal L_u=\langle ux|\pi\rangle$, we have
\begin{eqnarray}
\frac{2}{n} \mathcal L_{e_\alpha}^2&=& \frac{2}{n} \langle e_\alpha x|\pi\rangle^2 = \frac{2}{n} \langle e_\alpha |x\pi\rangle^2 = \frac{2}{n} \langle  \pi x |\pi x\rangle \cr
&=&  \frac{1}{2n} ||\tr x \pi +\tr \pi x||^2 \cr
&=& \frac{1}{2n} \bigl( (\tr x)^2 ||\pi||^2 +(\tr \pi)^2 ||x||^2 +2\tr x\tr \pi \langle x\mid \pi\rangle\bigr)\cr
&=& \frac{1}{2n^2} \bigl( (\tr x)^2 \tr \pi^2 +(\tr \pi)^2 \tr x^2 +2\tr x\tr \pi \tr(\pi x) \bigr)\cr
&=& \frac{(\tr x)^2}{2n^2} \bigl( \tr \pi^2 +3(\tr \pi)^2  \bigr).\nonumber
\end{eqnarray}
Since $\mathcal L_e=\langle x|\pi\rangle=\frac{1}{n}\tr\pi \tr x$,
\begin{eqnarray}
\mathcal X_e &=& \langle x\mid \pi^2\rangle+ {n\mu^2\over \tr x}=\langle \pi x\mid \pi\rangle+ {n\mu^2\over \tr x}\cr
&=& \frac{1}{2}\langle \tr \pi x+\tr x\pi \mid \pi\rangle+ {n\mu^2\over \tr x}\cr
&=& \frac{\tr x}{2n}(\tr \pi^2 +(\tr \pi)^2 )+ {n\mu^2\over \tr x}\nonumber
\end{eqnarray}
and $\mathcal Y_e =\langle e\mid x\rangle = \frac{1}{n}\tr x$, so
\[
\mathcal L_e^2+\mathcal X_e \mathcal Y_e = \frac{(\tr x)^2}{2n^2} \bigl( \tr \pi^2 +3(\tr \pi)^2  \bigr)+\mu^2.
\]
The primary quadratic relation is then clear. 

\item
The two identities can be obtained by taking the Poisson bracket of the primary quadratic relation with $\mathcal  X_e$ and $\mathcal  Y_e$ respectively. For example, since
\[
{2\over n} \sum\{\mathcal L_{e_\alpha}^2, \mathcal X_e\}  - \{\mathcal  L_e^2, \mathcal X_e\}- \{\mathcal  X_e  \mathcal Y_e, \mathcal X_e  \} =0,
\]
we have  
\[
{4\over n}  \mathcal L_{e_\alpha} \mathcal X_{e_\alpha}-2 \mathcal  L_e \mathcal X_e- \mathcal  X_e \cdot 2 \mathcal L_e =0
\]
or \( \mathcal L_{e_\alpha} \mathcal X_{e_\alpha} = n\mathcal X_e  \mathcal  L_e \).

\item The identity can be obtained by taking the Poisson bracket of the primary quadratic relation with $\mathcal  L_u$.

\item By taking the Poisson bracket of  \(\mathcal L_{e_\alpha} \mathcal X_{e_\alpha} = n\mathcal X_e  \mathcal  L_e \) with $\mathcal  X_e$, we have $ \mathcal  X_{e_\alpha}^2 = n \mathcal  X_e^2$. Similarly, we can prove $\mathcal  Y_{e_\alpha}^2 =n\mathcal  Y_e^2$. 

\item  The two identities can be obtained by taking the Poisson bracket of the identity in (iii) with $\mathcal  X_e$ and $\mathcal  Y_e$ respectively.

\item  The identity can be obtained by taking the Poisson bracket of the first identity in (ii) with $\mathcal  Y_e$ and then using the primary quadratic relation. 
 
\item Taking the Poisson bracket of the first identity in (v) with $\mathcal Y_u$ and then taking $u=e_\beta$ and summing over $\beta$, we get
\[
\frac{2}{n}\bigl(\mathcal Y_{L_{e_\alpha, e_\beta}(e_\beta)}\mathcal X_{e_\alpha} -2 \mathcal L_{e_\alpha, e_\beta}^2   \bigr) = 2\mathcal L_{e_\beta^2}\mathcal L_e+\mathcal X_{e_\beta}\mathcal Y_{e_\beta} - \mathcal Y_{e_\beta^2}\mathcal X_e-2 \mathcal L_{e_\beta}^2
\] 
or
\[
\frac{2}{n}\bigl(\mathcal Y_{e_\alpha e_\beta^2-L_{e_\beta}^2e_\alpha}\mathcal X_{e_\alpha} -2 \mathcal L_{e_\alpha, e_\beta}^2   \bigr) = 2n^2 \mathcal L_e^2+\mathcal X_{e_\beta}\mathcal Y_{e_\beta} - n^2\mathcal Y_e\mathcal X_e-2 \mathcal L_{e_\beta}^2
\] 
Since $e_\alpha^2 = n^2 e$ and $L_{e_\beta}^2=\frac{n^2}{2}(L_e+|e\rangle\langle e|)$ (see line 13 after equation (6.23) of Ref. \cite{meng2011}), we have $e_\alpha e_\beta^2-L_{e_\beta}^2e_\alpha=\frac{n^2}{2} (e_\alpha- \langle e|e_\alpha\rangle e) $, so, in the preceding equation,
\[
LHS = n\bigl(\mathcal Y_{e_\alpha} \mathcal X_{e_\alpha} -\mathcal X_e\mathcal Y_e\bigr)-\frac{4}{n}\mathcal L_{e_\alpha, e_\beta}^2.
\]
Therefore, we have
\[
-\frac{4}{n}\mathcal L_{e_\alpha, e_\beta}^2 = 2n^2 \mathcal L_e^2+(1-n)\mathcal X_{e_\alpha}\mathcal Y_{e_\alpha} +n(1- n)\mathcal Y_e\mathcal X_e-2 \mathcal L_{e_\alpha}^2.
\]
Upon using identities in (i) and (vi), the preceding identity becomes
\[
-\frac{4}{n}\mathcal L_{e_\alpha, e_\beta}^2 = 2n^2 \mathcal L_e^2+(1-n)n(\mathcal L_e^2+\mu^2) +n(1- n)\mathcal Y_e\mathcal X_e-n (\mathcal L_e^2+\mathcal X_e\mathcal Y_e-\mu^2)
\]
or
\[
\frac{4}{n^3}\mathcal L_{e_\alpha, e_\beta}^2 = -\mathcal L_e^2+\mathcal X_e\mathcal Y_e+\frac{n-2}{n}\mu^2.
\]

\end{enumerate}
\end{proof} 
\begin{rmk}
In Ref. \cite{meng2011}, the quantum version of these quadratic relations is worked out for 
un-magntized generalized Kepler problems associated with arbitrary simple Euclidean Jordan algebra.  
As far as we know, the quantum version of these quadratic relations for the MICZ-Kepler problems appeared first in Ref. \cite{BarutBohm}. The fact that these quadratic relations are consequences of a single quadratic relation was observed first in Ref. \cite{meng2011}.  Please also compare with the relevant part in Refs. \cite{mengzhang, mengLMS, mengJLT, meng2013b, BaiThesis}.
\end{rmk}

As a corollary of these quadratic relations, let us derive a formula connecting the Hamiltonian to the angular momentum and the Laplace-Runge-Lenz vector. From Ref. \cite{meng2014'} we know that the Hamiltonian is
 \[
 H={1\over 2} {\mathcal X_e\over \mathcal Y_e} - {1\over \mathcal Y_e},
 \]
the angular momentum is $\mathcal L_{e_\alpha, e_\beta}$, and Laplace-Runge-Lenz vector is
\[
\mathcal A_{e_\alpha} = {1\over 2}\left(\mathcal X_{e_\alpha} - \mathcal Y_{e_\alpha}  {\mathcal X_e\over \mathcal Y_e}\right)+{\mathcal Y_{e_\alpha} \over \mathcal Y_e}.
\]
\begin{corollary} Let $e_\alpha$ be an orthonormal basis for ${\mathrm H}_n(\mathbb C)$, $L^2={1\over 2}\sum_{\alpha, \beta} \mathcal L_{e_\alpha, e_\beta}^2$, and $A^2=-1+\sum_\alpha \mathcal A_{e_\alpha}^2$. Then the Hamiltonian $H$ satisfies the relation 
\begin{eqnarray}\label{HLA}
-2H\left (L^2- \frac{n^2(n-1)}{4}\mu^2\right ) = \left(\frac{n}{2}\right)^2(n-1-A^2).
\end{eqnarray}
\end{corollary}
\begin{proof} For simplicity, we shall hide the summation sign $\Sigma$ in the proof below. Since
$\mathcal A_{e_\alpha} = {1\over 2}\left(\mathcal X_{e_\alpha} - \mathcal Y_{e_\alpha}  {\mathcal X_e\over \mathcal Y_e}\right)+{\mathcal Y_{e_\alpha} \over \mathcal Y_e}$, using quadratic relations in the previous theorem, we have
\begin{eqnarray}
\mathcal A_{e_\alpha}^2 &=& n + {1\over 4}\mathcal X_{e_\alpha}^2-{1\over 2}\mathcal X_{e_\alpha}\mathcal Y_{e_\alpha}{\mathcal X_e\over \mathcal Y_e}+{n\over 4}\mathcal X_e^2+ \mathcal X_{e_\alpha} {\mathcal Y_{e_\alpha} \over \mathcal Y_e} - n \mathcal X_e\quad \mbox{using $\mathcal Y_{e_\alpha}^2=n\mathcal Y_e$}\cr
&=& n + {n\over 2}  \mathcal X_e^2 + \mathcal X_{e_\alpha}\mathcal Y_{e_\alpha}\bigl(-{\mathcal X_e\over 2\mathcal Y_e} +\frac{1}{\mathcal Y_e} \bigr)- n \mathcal X_e\quad\mbox{using $\mathcal X_{e_\alpha}^2=n\mathcal X_e^2$}\cr
&=& n + (n\mathcal Y_e\mathcal X_e - \mathcal X_{e_\alpha}\mathcal Y_{e_\alpha} )H\quad \mbox{using $H=\frac{\mathcal X_e}{2\mathcal Y_e}-\frac{1}{\mathcal Y_e}$}\cr
&=& n + n(\mathcal Y_e\mathcal X_e -\mathcal L_e^2-\mu^2 )H \quad \mbox{using identity (vi) in Theorem \ref{QRelations}}\cr
&=& n + n\left(\frac{8}{n^3}L^2-\frac{2(n-1)}{n}\mu^2 \right)H\quad \mbox{using identity (vii) in Theorem \ref{QRelations}}\cr
&=& n + \frac{8}{n^2} H \left(L^2-\frac{n^2(n-1)}{4}\mu^2 \right)\cr
\end{eqnarray}
Since $A^2=-1+\sum \mathcal A_{e_\alpha}^2$, we are done.
\end{proof}

\begin{rmk}
If $n=2$, Eq. \eqref{HLA} becomes
\[-2H\left (L^2-\mu^2\right ) = (1-A^2),\]
i.e., the formula given by Eq. (2.8) of Ref. \cite{meng2012}. Please also compare with Eq. (6.10) in Ref. \cite{meng2013b}.
\end{rmk}

\end{document}